\documentclass[11pt]{article}
\usepackage{amsmath}
\usepackage{amssymb}
\usepackage{theorem}
\usepackage{euscript}
\usepackage{pstricks}
\newtheorem{theorem}{Theorem}[section]
\newtheorem{lemma}[theorem]{Lemma}
\newtheorem{corollary}[theorem]{Corollary}
\newtheorem{remark}[theorem]{Remark}
\numberwithin{equation}{section}
\newtheorem{definition}[theorem]{Definition}

\title{\sffamily Lower bounds for the integration error \\
 for multivariate functions with 
mixed smoothness and \\
optimal Fibonacci cubature for functions on the square}
\author{Dinh D\~ung$^a$\footnote{Corresponding author. Email:
dinhzung@gmail.com}, Tino Ullrich$^b$ \\\\
$^a$ Vietnam National University, Hanoi, Information Technology Institute \\
144, Xuan Thuy, Hanoi, Vietnam  \\\\
$^b$Hausdorff-Center for Mathematics and Institute for Numerical Simulation \\
 53115 Bonn, Germany}
\date{\ttfamily January 28, 2014 -- Version 13}
\def\II{{\mathbb I}}

\def\ZZ{{\mathbb Z}}
\def\NN{{\mathbb N}}
\def\RR{{\mathbb R}}
\def\GG{{\mathbb G}}

\def\TT{{\mathbb T}}
\def\IId{{\mathbb I}^d}

\def\RRd{{\mathbb R}^d}
\def\TTd{{\mathbb T}^d}
\def\GGd{{\mathbb G}^d}

\def\ZZdp{{\mathbb Z}^d_+}

\def\Ba{B^\alpha_{p,\theta}}
\def\BaG{B^\alpha_{p,\theta}(\GGd)}
\def\BaT{B^\alpha_{p,\theta}(\TTd)}
\def\BaI{B^\alpha_{p,\theta}(\IId)}
\def\Ua{U^\alpha_{p,\theta}}
\def\UaT{U^\alpha_{p,\theta}(\TTd)}

\def\UaG{U^\alpha_{p,\theta}(\GGd)}
\def\Int{\operatorname{Int}}

\newcommand{\re}{\ensuremath{\mathbb{R}}}\newcommand{\N}{\ensuremath{\mathbb{N}}
}
\newcommand{\zz}{\ensuremath{\mathbb{Z}}}\newcommand{\C}{\ensuremath{\mathbb{C}}
}
\newcommand{\tor}{\ensuremath{\mathbb{T}}}

\newcommand{\Z}{{\ensuremath{\zz}^d}}

\newcommand{\R}{\ensuremath{{\re}^d}}

\newcommand{\cf}{{\mathcal F}}
\begin{document}
\maketitle

\begin{abstract}
We prove lower bounds for the error of optimal cubature formulae for $d$-variate functions from Besov spaces of mixed smoothness $B^{\alpha}_{p,\theta}(\GGd)$ in the case  $0 < p, \theta \le \infty$ and 
$\alpha > 1/p$, where $\GGd$ is either the 
$d$-dimensional torus $\TTd$ or the $d$-dimensional unit cube $\IId$. We prove upper bounds 
for QMC methods of integration on the Fibonacci lattice for bivariate periodic functions from $\Ba(\TT^2)$ in the case $1\leq p \leq \infty$, $0 < \theta \leq \infty$, $\alpha>1/p$. A non-periodic modification of this classical formula yields upper bounds for $\Ba(\II^2)$ if $1/p<\alpha<1+1/p$. In combination these results yield the correct asymptotic error of optimal cubature formulae for functions from $\Ba(\GG^2)$ and indicate that a corresponding result is most likely also true in case $d>2$.  This is compared to the correct asymptotic of optimal cubature formulae on Smolyak grids which results in the observation that any cubature formula on Smolyak grids is never optimal for the general setting.



\end{abstract}

\section{Introduction}

This paper deals with optimal cubature formulae of functions with mixed smoothness defined either
on the $d$-cube $\IId:=[0,1]^d$ or the $d$-torus $\TTd=[0,1]^d$,
where in each component interval $[0,1]$ the points $0$ and $1$
are identified. Functions defined on $\TTd$ can be also considered as functions on $\RRd$ which are $1$-periodic in
each variable.
A general cubature formula is given by
\begin{equation}\label{1.1}
\Lambda_n(X_n,f) := \sum_{x^j\in X_n} \lambda_j f(x^j)
\end{equation}
and supposed to compute a good approximation of the integral
\begin{equation}
I(f) :=
\int_{[0,1]^d} f(x)\,dx
\end{equation}
 within a reasonable computing time. The discrete set
$X_n=\{x^j\}_{j=1}^n$ of $n$ integration knots in $[0,1]^d$ and the vector of weights $\Lambda_n =
(\lambda_1,...,\lambda_n)$ with the $\lambda_j \in \RR$ are fixed in
advance for a class $F_d$ of $d$-variate functions $f$ on $\GGd$, where
$\GGd$ denotes either $\TTd$ or $\IId$. If the weight sequence is constant $1/n$, i.e., if $\Lambda_n = (1/n,...,1/n)$
then we speak of a quasi-Monte-Carlo method (QMC) and we denote
$$
    I_n(X_n,f):= \Lambda_n(X_n,f)\,.
$$
The worst-case error of an optimal cubature formula with respect to the class $F_d$ is given by
\begin{equation}\label{f62}
  \Int_n(F_d):=\inf\limits_{X_n,\Lambda_n}\sup\limits_{f\in F_d}
|I(f)-\Lambda_n(X_n,f)|\,,\quad n\in \N\,.
\end{equation}
\indent Our main focus lies on integration in Besov-Nikol'skij spaces $B^{\alpha}_{p,\theta}(\GGd)$ of mixed smoothness $\alpha$, where $0 < p,\theta \leq \infty$ and $\alpha>1/p$. Let $\UaG$ denote the unit ball in $B^{\alpha}_{p,\theta}(\GGd)$. The present paper is a continuation
of the second author's work \cite{Ul12_3} where optimal cubature of bivariate functions from 
$U^{\alpha}_{p,\theta}(\tor^2)$ on Hammersley type point sets has been studied. Indeed, here we investigate the asymptotic of the quantity $\Int_n(\UaG)$ where, in contrast to \cite{Ul12_3}, the smoothness $\alpha$ can now be larger or equal to $2$. This by now classical research topic goes back to the work of
Korobov \cite{Ko59}, Hlawka \cite{Hl62}, and Bakhvalov \cite{Ba63} in the 1960s.
In contrast to the quadrature of univariate functions, where equidistant point
grids lead to optimal formulas, the multivariate problem is much more involved.
In fact, the choice of proper sets $X_n \subset \TTd$ of integration knots is
connected with deep problems in number
theory, already for $d=2$.\\
\indent
 Spaces of mixed smoothness have a
long history in the former Soviet Union, see \cite{Am76, Di86, ScTr87, Te93} and the references therein,
and continued attracting significant interest also in the last 5 years
\cite{Vy06, TDiff06,Di11}. Temlyakov \cite{Te91} studied optimal cubature in
the related Sobolev spaces $W^{\alpha}_{p}(\tor^2)$ of mixed smoothness as well
as in Nikol'skij spaces
$B^{\alpha}_{p,\infty}(\tor^2)$ by using formulae based on Fibonacci
numbers (see also \cite[Thm.\ IV.2.6]{Te93}). This highly nontrivial idea goes back to Bakhvalov \cite{Ba63} and indicates once
more the deep connection to number theoretical issues. In the present paper, we extend these results to values
$\theta <\infty$ and prove
the relation
\begin{equation} \label{f00}
      \Int_n(\Ua(\TT^2)) \asymp n^{-\alpha}(\log n)^{(1-1/\theta)_+}\,,\quad n\in \N\,.
\end{equation}
As one would expect, also Fibonacci
quasi-Monte-Carlo methods are optimal and yield the correct asymptotic of $\Int_n(\Ua(\TT^2))$ in \eqref{f00}. Note
that the case $0 < \theta \le 1$ is not excluded and the $\log$-term disappears.
Thus, the optimal integration error decays as quickly as in the univariate
case. In fact, this represents one of the motivations to consider the third index
$\theta$. Note that the Fibonacci cubature formulae so far do not have a proper extension to $d$ dimensions.
Hence, the method in Corollary \ref{Corollary:UpperBound[d=2]} below does not help for general $d>2$. 
For a partial result in case $1/p <\alpha \le 1$ and arbitrary $d$ let us refer to \cite{Ma13_2, Ma13_1, MaDiss12}.\\
\indent Not long ago, Triebel \cite[Chapt.\ 5]{Tr10} proved that if  $1 \leq p,\theta \leq \infty$ and
$1/p < \alpha < 1 + 1/p$, then
\begin{equation}\label{f04}
 n^{-\alpha}(\log n)^{1-1/\theta}\lesssim
\Int_n(\Ua(\II^2)) \lesssim
n^{-\alpha}(\log n)^{\alpha+1-1/\theta}\,,\quad n\in \N\,,
\end{equation}
by using integration knots from Smolyak grids \cite{Sm63}. The gap between upper and lower bounds in
\eqref{f04} has been recently closed by the second named author \cite{Ul12_3}. Let us point out that, although 
we have established the correct asymptotic \eqref{f00} for $\Int_n(\Ua(\TT^2))$ in the periodic setting for all $\alpha>1/p$, 
it is still open for $\Int_n(\Ua(\II^2))$ and large $\alpha \geq 1+1/p$.\\
\indent One of the main contributions of this paper is the lower bound 
\begin{equation} \label{f06}
   \Int_n(\UaG) \ \gtrsim \ n^{-\alpha}(\log n)^{(d-1)(1-1/\theta)_{+}}\,,\quad n\in \N\,,
\end{equation}
for general $d$ and all $\alpha>1/p$. As the main tool we use the $B$-spline representations of functions from Besov spaces with
mixed smoothness based on the first author's work \cite{Di11}. To establish \eqref{f00} we exclusively used the Fourier analytical characterization of bivariate Besov spaces of mixed smoothness in terms of a decomposition of the frequency domain. \\
\indent The results in the present paper \eqref{f00} and \eqref{f06} as well other particular results in 
\cite{Te93}, \cite{Ma13_2, Ma13_1, MaDiss12} lead to the strong
conjecture that
\begin{equation} \label{f05}
   \Int_n(\UaG) \asymp n^{-\alpha}(\log n)^{(d-1)(1-1/\theta)_{+}}\,,\quad n\in \N\,,
\end{equation}
for all $\alpha>1/p$ and all $d>1$. In fact, the main open problem is the upper bound in \eqref{f05} for $d>2$ and $ \alpha>1$. In a special case, namely for $p=\theta=\infty$, $0 <\alpha <1$ and $\GGd = \TTd$, the conjecture \eqref{f05} has been already proved by Temlyakov \cite[Thms.\ IV.1.1 and IV.3.3]{Te93} by showing that the Korobov cubature formulae are optimal. Recently, Markhasin \cite{Ma13_2, Ma13_1, MaDiss12} proved \eqref{f05} in case $1/p<\alpha\leq 1$ for the slightly smaller classes $\Ua(\II^d)^\urcorner$ with vanishing
boundary values on the ``upper'' and ``right'' boundary faces of $\II^d = [0,1]^d$. \\
\indent Moreover, in the present paper we are concerned with the problem of optimal cubature on so-called Smolyak grids \cite{Sm63}, given by 
\begin{equation}\label{eq53}
G^d(m)
:= \bigcup\limits_{k_1+...+k_d \leq m} I_{k_1} \times ... \times I_{k_d}
\end{equation}
where $I_k:=\{2^{-k}\ell:\ell = 0,...,2^k-1\}$. If $\Lambda_m = (\lambda_{\xi})_{\xi\in G^d(m)}$, we consider the cubature formula
$\Lambda^s_m(f) := \Lambda_m(G^d(m),f)$ on Smolyak grids $G^d(m)$ given by
\begin{equation*}
\Lambda^s_m(f)
\ = \
\sum_{\xi \in G^d(m)}  \lambda_{\xi} f(\xi).
\end{equation*}
The quantity of optimal cubature $\Int^s_n(F_d)$ on Smolyak grids $G^d(m)$ is then introduced by
\begin{equation} \label{def:Int^s_n}
\Int^s_n(F_d)
\ := \ \inf_{|G^d(m)| \le n, \, \Lambda_m} \  \sup_{f \in F_d} \, |f - \Lambda^s_m(f)|.
\end{equation}
For $\ 0 < p, \theta \le \infty$ and $\alpha > 1/p$, we obtain the correct asymptotic behavior
\begin{equation} \label{Asymp[Int^s_n(UaG)]}
\Int^s_n(\UaG)
\ \asymp \
n^{-\alpha} (\log n)^{(d-1)(\alpha + (1 - 1/\theta)_+)}\,,\quad n\in \N\,,
\end{equation}
which, in combination with \eqref{f00}, shows that cubature formulae $\Lambda^s_m(f)$ on Smolyak grids $G^d(m)$ can never be optimal for $\Int_n(\Ua(\TT^2))$. The upper bound of \eqref{Asymp[Int^s_n(UaG)]} follows from results on sampling recovery in the $L_1$-norm proved in \cite{Di11}. For surveys and recent results on
sampling recovery on Smolyak grids see, for example, \cite{BG04},
\cite{Di11}, \cite{SU07}, and \cite{SU11}. To obtain the lower bound we construct test functions based on $B$-spline representations of functions from $\BaT$. 
In fact, it turns out that the errors of sampling recovery and numerical integration on Smolyak grids asymptotically coincide. \\
\indent The paper is organized as follows. In Section 2 we introduce the relevant
Besov spaces $\BaG$ and our main tools, their B-spline representation as well as a Fourier analytical characterization of bivariate Besov spaces $\Ba(\TT^2)$ in terms of a dyadic decomposition of the frequency domain. Section 3 deals with
the cubature of bivariate periodic and non-periodic functions from $\Ua(\GG^2)$ on the Fibonacci lattice. In particular, 
we prove the upper bound of \eqref{f00}, whereas in Section 4 we establish the lower bound \eqref{f06} for general $d$ 
and all $\alpha>1/p$. Section 5 is concerned with the relation 
\eqref{Asymp[Int^s_n(UaG)]} as well the asymptotic behavior of the quantity of optimal sampling recovery on Smolyak grids. \\
\indent {\bf Notation.} Let us introduce some common notations which are used in the present paper.
 As usual, $\N$ denotes the natural numbers, $\zz$ the integers and
$\re$ the real numbers. The set $\zz_+$ collects the nonnegative integers, sometimes we also use $\N_0$. 
We denote by $\tor$ the torus represented as the interval $[0,1]$ with identification of the end points.
For a real number $a$ we put $a_+ := \max\{a,0\}$.
The symbol $d$ is always reserved for the dimension in $\Z$, $\R$, $\N^d$, and $\tor^d$.
For $0<p\leq \infty$ and $x\in \R$ we denote $|x|_p = (\sum_{i=1}^d |x_i|^p)^{1/p}$ with the
usual modification in case $p=\infty$. The inner product between two vectors $x,y\in \re^d$ is denoted 
by $x\cdot y$ or $\langle x,y\rangle$. In particular, we have  $|x|_2^2 = x\cdot x = \langle x,x\rangle$.
For a number $n\in \N$ we set $[n] = \{1,..,n\}$.
If $X$ is a (quasi-)Banach space, the norm
of an element $f$ in $X$ will be denoted by $\|f\|_X$. For real numbers $a, b>0$ we use the notation $a \lesssim b$ if it exists a constant $c>0$ (independent of the relevant parameters) such that $a \leq cb$. Finally, $a \asymp b$ means $a\lesssim b$  and $b \lesssim a$.

\section{Besov spaces of mixed smoothness}

Let us define Besov spaces of mixed smoothness $\BaG$, where $\GGd$ denotes either $\TTd$ or $\IId$. 
In order to treat both situations, periodic and non-periodic spaces, simultaneously, we use the classical definition via mixed moduli of smoothness. Later we will add the Fourier analytical characterization for spaces on $\TT^2 $ in terms of a decomposition in frequency domain. Let us first recall the basic concepts. For univariate functions $f:[0,1] \to \C$ the $\ell$th difference operator $\Delta_h^\ell$ is defined by
\begin{equation*}
\Delta_h^\ell(f,x) := \left\{\begin{array}{rcl}
\sum_{j =0}^\ell (-1)^{\ell - j} \binom{\ell}{j} f(x + jh) &:& x+\ell h \in [0,1],\\[1.5ex]
0&:& \mbox{otherwise}\,.
\end{array}\right.
\end{equation*}
Let $e$ be any subset of
$[d]$. For multivariate functions $f:\II^d\to \C$ and $h\in \RR^d$
the mixed $(\ell,e)$th difference operator $\Delta_h^{\ell,e}$ is defined by
\begin{equation*}
\Delta_h^{\ell,e} := \
\prod_{i \in e} \Delta_{h_i}^\ell\quad\mbox{and}\quad \Delta_h^{\ell,\emptyset} =  \operatorname{Id},
\end{equation*}
where $\operatorname{Id}f = f$ and the univariate operator
$\Delta_{h_i}^\ell$ is applied to the univariate function $f$ by considering $f$ as
a function of variable $x_i$ with the other variables kept fixed. In case $d=2$ we slightly simplify the notation and
use $\Delta^{\ell}_{(h_1,h_2)} := \Delta^{\ell,\{1,2\}}_h$, $\Delta^{\ell}_{h_1,1}:=\Delta^{\ell,\{1\}}_h$, and
$\Delta^{\ell}_{h_2,2}:=\Delta^{\ell,\{2\}}_h$\,.

For $0 < p \le \infty$, denote by
$L_p(\GGd)$ the quasi-normed space
of functions on $\GGd$ with finite $p$th integral quasi-norm
$\|\cdot\|_p:= \|\cdot\|_{L_p(\GGd)}$ if $0 < p < \infty$, and sup-norm 
$\|\cdot\|_{\infty}:= \|\cdot\|_{L_\infty(\GGd)}$ if $p = \infty$.
Let
\begin{equation*}
\omega_{\ell}^e(f,t)_p:= \sup_{|h_i| < t_i, i \in
e}\|\Delta_h^{\ell,e}(f)\|_{p}\quad,\quad t \in {\II}^d,
\end{equation*}
be the mixed $(\ell,e)$th modulus of smoothness of $f \in L_p(\GGd)$ (in particular, $\omega_{\ell}^{\emptyset}(f,t)_p = \|f\|_{p}$)\,. Let us turn to the definition of the Besov classes $B^{\alpha}_{p,\theta}(\GGd)$. For $0 <  p, \theta \le \infty$,
$\alpha > 0$ and $\ell > \alpha$
we introduce the quasi-semi-norm
$|f|_{B_{p, \theta}^{\alpha,e}(\GGd)}$ for functions $f \in L_p(\GGd)$ by
\begin{equation*} \label{BesovSeminorm}
|f|_{B_{p, \theta}^{\alpha,e}(\GGd)}:=
\left\{\begin{array}{rcl}
\Big(\int_{{\II}^d} \Big[\prod_{i \in e} t_i^{-\alpha}
\omega_{\ell}^e(f,t)_p \Big]^ \theta \prod_{i \in e} t_i^{-1}dt
\Big)^{1/\theta} &:& \theta <\infty\,,\\[2ex]
\sup_{t \in {\II}^d} \ \prod_{i \in e} t_i^{-\alpha}\omega_{\ell}^e(f,t)_p &:& \theta = \infty\,
\end{array}\right.
\end{equation*}
(in particular, $|f|_{B_{p, \theta}^{\alpha,\emptyset}(\GGd)} =
\|f\|_{p}$).

\begin{definition}  For $0 <  p, \theta \le \infty$ and $0 < \alpha < \ell$ the
Besov space
$\BaG$ is defined as the set of  functions $f \in L_p(\GGd)$
for which the Besov quasi-norm $\|f\|_{\BaG}$ is finite.
The  Besov quasi-norm is defined by
\begin{equation*}
\|f\|_{\BaG}
:= \
 \sum_{e \subset [d]} |f|_{B_{p, \theta}^{\alpha,e}(\GGd)}.
 \end{equation*}
 The space of periodic functions $\BaT$ can be considered as a subspace of
$\BaI$.
\end{definition}

\subsection{B-spline representations on $\IId$}

For a given natural number $r \ge 2$ let  $N$ be the
cardinal B-spline of order $r$ with support $[0,r]$, i.e.,
$$
        N(x) = \underbrace{(\chi \ast \cdots \ast \chi)}_{r-\mbox{fold}}(x)\quad,\quad x\in \re\,,
$$
where $\chi(x)$ denotes the indicator function of the interval $[0,1]$\,.
We define the integer translated dilation $N_{k,s}$ of $N$ by
\begin{equation*}
N_{k,s}(x):= \ N(2^k x - s), \ k \in {\ZZ}_+, \ s \in \ZZ,
\end{equation*}
and the $d$-variate B-spline $N_{k,s}(x)$, $k \in {\ZZ}^d_+, \ s \in {\ZZ}^d$,  by
\begin{equation} \label{def:Mixed[N_{k,s}]}
N_{k,s}(x):=  \ \prod_{i=1}^d N_{k_i, s_i}( x_i)\quad,\quad x\in \R\,.
\end{equation}
Let $J^d(k):= \{s \in \ZZdp: -r < s_j < 2^{k_j}, \ j \in [d]\}$ be the set of
$s$ for which $N_{k,s}$ do not vanish identically on $\IId$, and denote by $\Sigma^d(k)$ the span of the B-splines
$N_{k,s}, \ s \in J^d(k)$.
If $0 < p \le \infty,$
for all $k \in {\ZZ}^d_+$ and all $g \in \Sigma^d(k)$ such that
\begin{equation} \label{def:StabIneq}
g = \sum_{s \in J^d(k)} a_s N_{k,s},
\end{equation}
there is the quasi-norm equivalence
\begin{equation} \label{eq:StabIneq}
 \|g\|_p
\ \asymp \ 2^{- |k|_1/p}\Big(\sum_{s \in J^d(k)}| a_s |^p \Big)^{1/p}.
\end{equation}
with the corresponding change when $p= \infty$.

We extend the notation $x_+:= \max\{0,x\}$ to vectors $x \in \RR^d$ by putting $x_+:= ((x_1)_+, ..., (x_d)_+)$\,.
Furthermore, for a subset $e \subset \{1,...,d\}$ we define the subset $\ZZdp(e) \subset \ZZ^d$ by $\ZZdp(e):= \{s \in \ZZdp: s_i = 0 , \ i \notin e\}$.
For a proof of the following lemma we refer to \cite[Lemma 2.3]{Di11}.

\begin{lemma} \label{Lemma:IneqMixedomega_{r}^e}
Let $0 < p \le \infty, \ 0 < \tau \le \min\{p,1\}, \ \delta = \min \{r, r - 1 +
1/p\}$. If the continuous function $g$ on $\IId$ is represented by the series $g \ = \sum_{k \in {\ZZ}^d_+} \ g_k$
with convergence in $C(\IId)$, where $g_k \in \Sigma_r^d(k)$, then we have for any $\ell \in {\ZZ}^d_+(e)$,
\begin{equation*}
\omega_{r}^e(g,2^{-\ell})_p
\ \le \
C   \Big( \sum_{k\in {\ZZ}^d_+}
 \Big[ 2^{-\delta |(\ell-k)_+|_1} \| g_k\|_p \Big]^\tau \Big)^{1/\tau}\,,
\end{equation*}
whenever the sum on the right-hand side is finite. The constant $C$ is independent of  $g$ and $\ell$\,.
\end{lemma}

As a next step, we obtain as a consequence of Lemma \ref{Lemma:IneqMixedomega_{r}^e} the following result. Its proof
is similar to the one in \cite[Theorem 2.1(ii)]{Di11} (see also \cite[Lemma 2.5]{Di13}). The main tool is an application of the discrete Hardy inequality, see \cite[(2.28)--(2.29)]{Di11}.

\begin{lemma} \label{lemma[InverseRepresTheorem]}
Let $\ 0 < p, \theta \le \infty$ and $0 < \alpha < \min\{r, r - 1 + 1/p\}$.
Let further $g$ be a continuous function on $\IId$ which is represented by a series
\begin{equation} \nonumber
g \ = \
\sum_{k \in {\ZZ}^d_+} \sum_{s \in J^d(k)} c_{k,s} N_{k,s}
\end{equation}
with convergence in $C(\IId)$, and the coefficients $c_{k,s}$ satisfy the condition
\begin{equation} \nonumber
B(g):= \
\Big(\sum_{k \in \ZZdp}
2^{\theta(\alpha - 1/p)|k|_1}\Big[\sum_{s \in J^d(k)}| c_{k,s} |^p
\Big]^{\theta/p} \Big)^{1/\theta} \ < \ \infty
\end{equation}
with the change to sup for $\theta = \infty$. Then $g$ belongs the  space $\BaI$ and
\begin{equation} \nonumber
\|g\|_{\BaI}
\ \lesssim \
B(g).
\end{equation}
\end{lemma}

\subsection{The tensor Faber basis in two dimensions}
Let us collect some facts about the important special case $r=2$ of the cardinal B-spline system. The resulting
system is called ``tensor Faber basis". In this subsection we will mainly focus on a converse statement
to Lemma \ref{lemma[InverseRepresTheorem]} in two dimensions. \\
\indent To simplify notations let us introduce the set $\N_{-1} = \N_0 \cup \{-1\}$. Let further
$D_{-1} := \{0,1\}$ and $D_j := \{0,...,2^j-1\}$ if $j\geq 0$\,. Now we define for $j\in \N_{-1}$ and $m\in D_j$
\begin{equation}
    v_{j,m}(x)=\left\{\begin{array}{lcl}
	   2^{j+1}(x-2^{-j}m)&:&2^{-j}m \leq x \leq 2^{-j}m+2^{-j-1},\\[1.5ex]
	   2^{j+1}(2^{-j}(m+1)-x)&:&2^{-j}m+2^{-j-1}\leq x \leq 2^{-j}(m+1),\\[1.5ex]
	   0&:& \mbox{otherwise}\,.
        \end{array}\right.
\end{equation}
Let now $j = (j_1,j_2) = \N_{-1}^2$, $D_j = D_{j_1} \times D_{j_2}$
and $m = (m_1,m_2) \in D_j$. The bivariate (non-periodic) Faber basis functions result from a tensorization
of the univariate ones, i.e.,
\begin{equation}\label{f19nonp}
    v_{(j_1,j_2),(m_1,m_2)}(x_1,x_2)=\left\{\begin{array}{lcl}
	   v_{m_1}(x_1)v_{m_2}(x_2)&:&j_1=j_2 = -1,\\[1.5ex]
	   v_{m_1}(x_1)v_{j_2,m_2}(x_2)&:&j_1=-1,j_2\in \N_0,\\[1.5ex]
	   v_{j_1,m_1}(x_1)v_{m_2}(x_2)&:&j_1\in \N_0, j_2 = -1,\\[1.5ex]
	   v_{j_1,m_1}(x_1)v_{j_2,m_2}(x_2)&:&j_1,j_2\in \N_0\,,
        \end{array}\right.
\end{equation}
see also \cite[3.2]{Tr10}. For every continuous bivariate function $f\in C(\II^2)$ we have the representation
\begin{equation}\label{eq34}
  f(x) = \sum\limits_{j\in \N^2_{-1}} \sum\limits_{m\in D_j}
  D^2_{j,m}(f)v_{j,m}(x)\,,
\end{equation}
where now
$$
  D^2_{j,k}(f) = \left\{\begin{array}{lcl}
    f(m_1,m_2)&:&j = (-1,-1),\\[1.5ex]
    -\frac{1}{2}\Delta^{2}_{2^{-j_1-1},1}(f,(2^{-j_1}m_1,0))&:& j=(j_1,-1),\\[1.5ex]
    -\frac{1}{2}\Delta^{2}_{2^{-j_2-1},2}(f,(0, 2^{-j_2}m_2))&:& j=(-1,j_2),\\[1.5ex]
\frac{1}{4}\Delta^{2,2}_{(2^{-j_1-1}, 2^{-j_2-2})}(f,(2^{-j_1}m_1,2^{-j_2}m_2))&:& j
= (j_1,j_2)\,.
        \end{array}\right.
$$
The following result states the converse inequality to Lemma \ref{lemma[InverseRepresTheorem]} in the particular situation of the bivariate tensor Faber
basis. For the proof we refer to \cite[Thm.\ 4.1]{Di11} or \cite[Thm.\ 3.16]{Tr10}. Note that in the latter reference the additional stronger restriction $1/p<\alpha<\min\{2, 1+1/p\}$ comes into play. However, it is not needed for this relation.

\begin{lemma} \label{conv} Let $0<p,\theta\leq \infty$ and $1/p<\alpha<2$. Then we have for any 
$f\in B^{\alpha}_{p,\theta}(\II^2)$,
\begin{equation}\label{f66}
    \Big[\sum\limits_{j\in \N_{-1}^2}2^{|j|_1(\alpha-1/p)\theta}\Big(\sum\limits_{k\in
    D_j}|D^2_{j,k}(f)|^p\Big)^{\theta/p} \Big]^{1/\theta}  \lesssim \|f\|_{B^{\alpha}_{p,\theta}(\II^2)}.
\end{equation}
\end{lemma}

The following lemma is a periodic version of Lemma \ref{lemma[InverseRepresTheorem]} for the tensor Faber basis. For a proof we refer to \cite[Prop.\ 3.6]{Ul12_3}.

\begin{lemma}\label{contvsdisc} Let $0<p,\theta \leq \infty$ and
$1/p<\alpha<1+\min\{1/p,1\}$. Then we have for all $f \in C(\TT^2)$,
$$
    \|f\|_{B^{\alpha}_{p,\theta}(\TT^2)} \lesssim \Big[\sum\limits_{j\in \N_{-1}^2}2^{|j|_1(\alpha-1/p)\theta}\Big(\sum\limits_{k\in
    D_j}|D^2_{j,k}(f)|^p\Big)^{\theta/p} \Big]^{1/\theta}\,
$$
whenever the right-hand side is finite. Moreover, if the right-hand side is finite, we have that $f\in B^{\alpha}_{p,\theta}(\tor^2)$\,.
\end{lemma}

\subsection{Decomposition of the frequency domain}

We consider the Fourier analytical characterization of bivariate Besov spaces
of mixed smoothness. The characterization comes
from a partition of the frequency domain. The following assertions
have counterparts also for $d>2$, see \cite{TDiff06}. Here, we will need it just for $d=2$.

    \begin{definition}\label{cunity} Let $\Phi(\re)$ be defined as the collection of all 
      systems
       $\varphi = \{\varphi_j(x)\}_{j=0}^{\infty} \subset C^{\infty}_0(\re)$
       satisfying
       \begin{description}
       \item(i) $\operatorname{supp}\,\varphi_0 \subset \{x:|x| \leq 2\}$\, ,
       \item(ii) $\operatorname{supp}\,\varphi_j \subset \{x:2^{j-1} \leq |x|
       \leq 2^{j+1}\}\quad,\quad j= 1,2,... ,$
       \item(iii) For all $\ell \in \N_0$ it holds
       $\sup\limits_{x,j}
       2^{j\ell}\, |D^{\ell}\varphi_j(x)| \leq c_{\ell} <\infty$\, ,
       \item(iv) $\sum\limits_{j=0}^{\infty} \varphi_j(x) = 1$ for all
       $x\in \re$.
       \end{description}
    \end{definition}

    \begin{remark}\label{speciald}\rm
      The class $\Phi(\re)$ is not empty. Consider the following example.
      Let $\varphi_0(x)\in C^{\infty}_0(\re)$ be smooth function with $\varphi_0(x) =
1$ on $[-1,1]$ and $\varphi_0(x) = 0$
      if $|x|>2$. For $j>0$ we define
      $$
         \varphi_j(x) = \varphi_0(2^{-j}x)-\varphi_0(2^{-j+1}x).
      $$
      Now it is easy to verify that the system $\varphi =
\{\varphi_j(x)\}_{j=0}^{\infty}$ satisfies (i) - (iv).
    \end{remark}
    \noindent Now we fix a system $\{\varphi_j\}_{j=0}^{\infty} \in \Phi(\re)$.
    For $j = (j_1,j_2) \in \zz^2$ let the building blocks $f_{j}$ be given by
    \begin{equation}\label{f2}
	\delta_j(f)(x) = \sum\limits_{k\in \zz^2}
	\varphi_{j_1}(k_1)\varphi_{j_2}(k_2)\hat{f}(k)e^{i 2\pi k\cdot x}\,,
    \end{equation}
    where we put $f_j = 0$ if $\min\{j_1,j_2\}<0$.
    \begin{lemma}\label{Fourier} Let $0<p,\theta \leq \infty$ and
    $\alpha \in \re$. Then $B^{\alpha}_{p,\theta}(\tor^2)$ is the collection of
all $f\in
    L_p(\tor^2)\cap L_1(\tor^2)$ such that
    \begin{equation}\label{f3}
	\|f|B^{\alpha}_{p,\theta}(\tor^2)\|:=\Big(\sum\limits_{j\in
\N_0^2}
2^{|j|_1\alpha \theta}\|\delta_j(f)\|_p^\theta\Big)^{1/\theta}
    \end{equation}
    is finite (usual modification in case $q=\infty$). Moreover, the
   quasi-norms $\|\cdot\|_{B^{\alpha}_{p,\theta}(\tor^2)}$ and
    $\|\cdot |B^{\alpha}_{p,\theta}(\tor^2)\|$ are equivalent.
    \end{lemma}
\begin{proof} For the bivariate case we refer to \cite[2.3.4]{ScTr87}. See
\cite{TDiff06} for the corresponding characterizations of Besov-Lizorkin-Triebel
spaces with dominating mixed smoothness on $\re^d$ and $\tor^d$.
\end{proof}

\section{Integration on the Fibonacci lattice}
\label{Fib}
In this section we will prove upper bounds for
$\Int_n(U^{\alpha}_{p,\theta}(\GG^2))$ which are realized by Fibonacci cubature formulas. If $\GG = \tor$ we obtain sharp results for all $\alpha>1/p$ whereas we need
the additional condition $1/p<r<1+1/p$ if $\GG = \II$. The restriction to $d=2$ is due the
concept of the Fibonacci lattice rule which so far does not have a proper
extension to $d>2$. The Fibonacci numbers given by
\begin{equation}\label{f31}
   b_0 = b_1 = 1~,\quad b_n = b_{n-1} + b_{n-2}~,\quad n\geq 2\,,
\end{equation}
play the central role in the definition of the associated integration lattice. In the sequel, the
symbol $b_n$ is always reserved for \eqref{f31}. For $n\in
\N$ we are going to study the Fibonacci cubature formula
\begin{equation}\label{f32}
   \Phi_n(f):=I_{b_n}(X_{b_n},f) = \frac{1}{b_n}\sum\limits_{\mu = 0}^{b_n-1}
f(x^{\mu})
\end{equation}
for a function $f\in C(\tor^2)$, where the lattice $X_{b_n}$ is given
by
\begin{equation}\label{eq41}
    X_{b_n} :=
\Big\{x^{\mu} = \Big(\frac{\mu}{b_n},\Big\{\mu\frac{b_{n-1}}{b_n}\Big\}\Big):\mu
= 0,...,b_n-1\Big\}~,\quad n\in \N\,.
\end{equation}
Here, $\{x\}$ denotes the fractional part, i.e., $\{x\} := x-\lfloor x \rfloor$
of the positive real number $x$. Note that $\Phi_n(f)$ represents a special
Korobov type \cite{Ko59} integration formula. The idea to use Fibonacci numbers
goes back to \cite{Ba63} and was later used by Temlyakov \cite{Te91} to study
integration in spaces with mixed smoothness (see also the recent contribution
\cite{BiTeYu11}). We will first focus on periodic functions and extend the results later
to the non-periodic situation.

\subsection{Integration of periodic functions}

We are going to prove the theorem
below which extends Temlyakov's results \cite[Thm.\ IV.2.6]{Te93} on the spaces
$B^{\alpha}_{p,\infty}(\tor^2)$, to the spaces $B^{\alpha}_{p,\theta}(\tor^2)$ with
$0 < \theta \leq \infty$. By using simple
embedding properties, our results below directly imply Temlyakov's earlier
results \cite[Thm.\ IV.2.1]{Te93}, \cite[Thm.\ 1.1]{BiTeYu11} on Sobolev spaces
$W^r_p(\tor^2)$. Let us denote by
$$
R_n(f) := \Phi_n(f)-I(f)
$$
the Fibonacci integration error.

\begin{theorem}\label{mainup} Let {$1\leq p \leq \infty$, $0 < \theta \leq \infty$} and $\alpha > 1/p$.
Then there exists a constant $c>0$ depending only on
$\alpha,p$ and $\theta$ such that
$$
\sup\limits_{f\in
U^{\alpha}_{p,\theta}(\tor^2)} |R_n(f)| \leq c\,b_n^{-\alpha}(\log
b_n)^{(1-1/\theta)_+}~,\quad n\in \N\,.
$$
 \end{theorem}

We postpone the proof of this theorem to Subsection \ref{Proof of Theorem ref{mainup}}.

\begin{corollary} \label{Corollary:UpperBound[d=2]}
Let $1\leq p \leq \infty$, $0 < \theta \leq \infty$ and $\alpha > 1/p$.
Then there exists a constant $c>0$ depending only on
$\alpha,p$ and $\theta$ such that
$$
  {\Int}_n(U^{\alpha}_{p,\theta}(\tor^2)) \leq c\,n^{-\alpha}(\log n)^{(1-1/\theta)_+}~,\quad n\in \N\,.
$$
\end{corollary}

\begin{proof} Fix $n\in \N$ and let $m\in \N$ such that $b_{m-1} < n \leq b_m$.
Put $U:=U^{\alpha}_{p,\theta}(\tor^2)$. Clearly, we have by Theorem \ref{mainup}
$$
 \Int_n(U) \leq \Int_{b_{m-1}}(U) \lesssim b_{m-1}^{-\alpha}(\log
b_{m-1})^{(1-1/\theta)_+} \leq n^{-\alpha}(\log n)^{(1-1/\theta)_+}\cdot
\Big(\frac{n}{b_{m-1}}\Big)^{\alpha}\,.
$$
By definition $n/b_{m-1} \leq b_m/b_{m-1}$. It is well-known that
$$
    \lim\limits_{m\to \infty} \frac{b_{m}}{b_{m-1}} = \tau\,,
$$
where $\tau$ represents the inverse Golden Ratio. The proof is
complete.
\end{proof}\\\\
Note that the case $0 < \theta \le 1$ is not excluded here. In this case
we obtain the upper bound $n^{-\alpha}$ without the $\log$ term. Consequently,
optimal cubature for this model of functions behaves like optimal
quadrature for $B^{\alpha}_{p,\theta}(\tor)$. We conjecture the same phenomenon for
$d$-variate functions. This gives one reason to vary the third index $\theta$
in $[0,\infty]$.

\subsection{Proof of Theorem \ref{mainup}} \label{Proof of Theorem ref{mainup}}

Let us divide the proof of Theorem \ref{mainup} into several steps. The first
part of the proof follows Temlyakov \cite[pages 220,221]{Te93}. To begin
with we will consider the integration error $R_n(f)$ for a
trigonometric polynomial $f$ on $\tor^2$. Let $f(x) = \sum_{k\in \zz^2}
\hat{f}(k)e^{2\pi ik\cdot x}$ be the Fourier series of $f$. Then clearly,
$\Phi_n(f) = \sum_{k\in \zz^2} \hat{f}(k) \Phi_n(e^{2\pi i k\cdot})$ and $I(f) =
\hat{f}(0)$. Therefore, we obtain
\begin{equation}\label{f56}
    R_n(f) = \sum\limits_{\substack{k\in \zz^2 \\ k\neq 0}} \hat{f}(k)
\Phi_n(k)~,
\end{equation}
where $\Phi_n(k):= \Phi_n(e^{2\pi i k\cdot})\,,k\in \zz^2$. By definition, we
have that
\begin{equation}\label{f33}
  \Phi_n(k) = \frac{1}{b_n}\sum\limits_{\mu = 0}^{b_n-1} e^{2\pi i
\mu\Big(\frac{k_1+b_{n-1}k_2}{b_n}\Big)}\,,
\end{equation}
and hence
\begin{equation}\label{f34}
    \Phi_n(k) = \left\{\begin{array}{rcl}
	   1&:&k\in L(n)\,,\\
	   0&:& k \notin L(n)\,,
        \end{array}\right.
\end{equation}
where
\begin{equation}\label{f35}
    L(n) = \{k = (k_1,k_2) \in \zz^2:k_1+b_{n-1}k_2 \equiv 0~(\mbox{mod }
b_n)\}\,.
\end{equation}
In fact, by the summation formula for the geometric series, we obtain from
\eqref{f33} that
$$
  \Phi_n(k) = \frac{1}{b_n}\frac{e^{2\pi i (k_1+b_{n-1}k_2)}-1}{e^{2\pi i
\Big(\frac{k_1+b_{n-1}k_2}{b_n}\Big)}-1} = 0
$$
in case $e^{2\pi i
\Big(\frac{k_1+b_{n-1}k_2}{b_n}\Big)} \neq 1$ or, equivalently, $k\notin L(n)$.
If $k\in L(n)$ then \eqref{f33} returns $\Phi_n(k) = 1$. Next we will study
the structure of the set $L(n)\setminus\{0\}$. Let us define the discrete sets
$\Gamma(\eta) \subset \zz^2$ by
$$
  \Gamma(\eta) = \{(k_1,k_2) \in \zz^2:\max\{1,|k_1|\}\cdot \max\{1,|k_2|\}\leq
\eta\}~,\quad \eta>0.
$$
The following two Lemmas are essentially Lemma IV.2.1 and Lemma IV.2.2,
respectively, in \cite{Te93}. They represent useful number theoretic
properties of the set $L(n)$. For the sake of completeness we
provide a detailed proof of Lemma \ref{L(n)_2} below.

\begin{lemma}\label{L(n)_1}
There exists a universal constant $\gamma>0$ such that
for every $n\in \N$,
\begin{equation}\label{f36}
  \Gamma(\gamma b_n) \cap
\big( L(n)\setminus \{0\}\big) = \emptyset\,.
\end{equation}
\end{lemma}

\begin{proof} See Lemma IV.2.1 in \cite{Te93}.
\end{proof}

\begin{lemma}\label{L(n)_2} For every $n\in \N$ the set $L(n)$ can be
represented in the form
\begin{equation}\label{f46}
  L(n) = \Big\{\big(ub_{n-2}-vb_{n-3}, u+2v):u,v \in \zz\Big\}.
\end{equation}
\end{lemma}

\begin{proof} Let $\tilde{L}(n) =
\big\{(ub_{n-2}-vb_{n-3},u+2v):u,v\in \zz\big\}$. \\\\
{\em Step 1.} We prove $\tilde{L}(n) \subset L(n)$. For $k\in \tilde{L}(n)$ we
have to show that $k_1+b_{n-1}k_2 = \ell b_n$ for some $\ell \in \zz$. Indeed,
$ub_{n-2}-vb_{n-3}+b_{n-1}(u+2v) = ub_n+vb_{n-2}+vb_{n-1} = b_n(u+v)$.\\\\
{\em Step 2.} We prove $L(n) \subset \tilde{L}(n)$. For $k = (k_1,k_2)\in
L(n)$ we have to find $u,v \in \zz$ such that the representation $k_1 =
ub_{n-2}-vb_{n-3}$ and $k_2 = u+2v$ holds true. Indeed, since $k\in L(n)$, we
have that $k_1+b_{n-1}k_2 = k_1 + (b_{n-3}+b_{n-2})k_2 = \ell b_n =
\ell(b_{n-3}+2b_{n-2})$ for some $\ell \in \zz$. The last identity implies
$k_1 = (\ell-k_2)b_{n-3}+(2\ell-k_2)b_{n-2}$. Putting $v = k_2-\ell$ and $u =
2\ell-k_2$ yields the desired representation.
\end{proof}

In the following, we will use a different argument than the one used by
Temlyakov to deal with the case $\theta = \infty$. We will modify the definition
of the functions $\chi_s$ introduced in \cite{Te93} before (2.37) on page 229.
This allows for the an alternative argument in order to incorporate the case
$p=1$ in the proof of Lemma \ref{chis} below. Let us also mention, that the
argument to establish the relation between (2.25) and (2.26) in  \cite{Te93} on
page 226 requires some additional work, see Step 3 of the proof of Lemma
\ref{chis} below.

For $s\in \N_0$ we define the discrete set $\rho(s) = \{k\in \zz: 2^{s-2} \leq
|k| < 2^{s+2}\}$ if $s\in \N$ and $\rho(s) = [-4,4]$ if $s=0$. Accordingly,
let $v_0(\cdot),v(\cdot),v_s(\cdot)$, $s\in \N$, be the piecewise linear
functions given by
\begin{equation}\nonumber
    v_0(t) = \left\{\begin{array}{rcl}
	   1&:& |t| \leq 2\,,\\[1.5ex]
	   -\frac{1}{2}|t|+2&:& 2<|t|\leq 4\\[1.5ex]
	   0&:& \mbox{otherwise}\,,
        \end{array}\right.
\end{equation}
$v(\cdot) = v_0(\cdot)-v_0(8\cdot)$, and $v_s(\cdot) = v(\cdot/2^s)$. Note
that $v_s$ is supported on $\rho_s$. Moreover, $v_0 \equiv 1$ on $[-2,2]$
and $v_s \equiv 1$ on $\{x:2^{s-1}\leq |x| \leq 2^{s+1}\}$. For $j
= (j_1,j_2)\in \N_0^2$ we put
$$
  \rho(j_1,j_2) = \rho(j_1) \times \rho(j_2)\quad\mbox{and} \quad v_j =
v_{j_1}\otimes v_{j_2}\,.
$$
We further
define the associated bivariate trigonometric polynomial
$$\chi_s(x) =
\sum_{k\in L(n)} v_s(k) e^{2\pi i k\cdot x}.
$$
Our next goal is to estimate $\|\chi_s\|_p$ for $1\leq p\leq \infty$.

\begin{lemma}\label{chis} Let $1\leq p\leq \infty$, $s \in \N_0^2$, and $n\in
\N$. Then
there is a constant $c>0$ depending only on $p$ such that
\begin{equation}\label{f48}
   \|\chi_s\|_p \leq c\Big(2^{|s|_1}/b_n\Big)^{1-1/p}\,.
\end{equation}
\end{lemma}

\begin{proof} {\em Step 1.} Observe first by Lemma \ref{L(n)_2} that
\begin{equation}\label{f51}
  \chi_s(x) = \sum\limits_{k\in \zz^2} v_s(B_nk)e^{2\pi \langle B_nk,x\rangle} =
\sum\limits_{k\in \zz^2} v_s(B_nk)e^{2\pi i \langle k,B_n^{\ast}x\rangle}\,,
\end{equation}
where
$$
    B_n = \left(\begin{array}{cc}
                   b_{n-1}&-b_{n-3}\\
                   1&2
                \end{array}\right)\,.
$$
It is obvious that $\det{B_n} = b_n$, which will be important in the sequel.
Clearly, if $\varepsilon>0$ is small enough we obtain
\begin{equation}
 \begin{split}
   \|\chi_s\|_\infty &\leq \sum\limits_{k\in \zz^2} v_s(B_nk) \leq
\sum\limits_{(x,y)\in B_n^{-1}(\rho(s))} 1 \\
&=
\frac{1}{4\varepsilon^2}\int\limits_{(B_n^{-1}(\rho(s)))_{\varepsilon}}d(x,y)
\lesssim \int\limits_{B_n^{-1}(Q(s))} d(x,y) =
\frac{1}{\det{B_n}}\int\limits_{Q(s)} d(u,v)\\
&\asymp \frac{2^{|s|_1}}{b_n}\,.
  \end{split}
\end{equation}
We used the notation $M_\varepsilon := \{z \in \re^2: \exists x\in M
\mbox{ such that } |x-z|_{\infty}<\varepsilon\}$ for a set $M \subset \re^2$ and
$Q(s)= \{x\in \re^2: 2^{s_j-3} \leq |x_j| < 2^{s_j+3}, j=1,2\}$ (modification
in case $s=0$). This proves \eqref{f48} in case $p=\infty$.\\\\
{\em Step 2.} Let us deal with the case $p=1$. By \eqref{f51} we have that
$\chi_s(\cdot) = \eta_s(B_n^{\ast}\cdot)$, where $\eta_s$
is the trigonometric polynomial given by
$$
    \eta_s(x) = \sum\limits_{k\in \zz^2} v_s(B_nk)e^{2\pi i k\cdot x}~,\quad
x\in
\tor^2\,.
$$
By Poisson's summation formula we infer that $\eta_s(\cdot) = \sum_{\ell\in
\zz^2} \cf^{-1}[v_s(B_n\cdot)](\cdot+\ell)$. Consequently,
$$
    \|\eta_s\|_1 = \int\limits_{\tor^2} |\eta_s(x)|\,dx \leq
\sum\limits_{\ell\in\zz^2}
\int\limits_{[0,1]^2}|\cf^{-1}[v_s(B_n\cdot)](x+\ell)|\,dx =
\|\cf^{-1}[v_s(B_n\cdot)]\|_{L_1(\re^2)}\,.
$$
The homogeneity of the Fourier transform implies then
\begin{equation}\label{f52}
\|\eta_s\|_1 =
\|\cf^{-1}v_s\|_{L_1(\re^2)} = \|\cf^{-1} v^s\|_{L_1(\re^2)}\,,
\end{equation}
where the function $v^s$ is one of the four possible tensor products of the
univariate functions $v_0$ and $v$ depending on $s$. Since these functions
are piecewise linear we obtain from \eqref{f52} the
relation $\|\eta_s\|_1 \lesssim 1$.\\\\
{\em Step 3.} It remains to show
$\|\eta_s(B_n^{\ast}\cdot)\|_1 \lesssim \|\eta_s\|_1$ which
implies \eqref{f48} in case $p=1$. In fact,
\begin{equation}\label{f54}
    \int\limits_{\tor^2} |\eta_s(B_n^{\ast} x)|\,dx =
\frac{1}{b_n}\int\limits_{B_n^{\ast}(0,1)^2}|\eta_s(x)|\,dx\,.
\end{equation}
Note that $B_n^{\ast}$ is a $2\times 2$ matrix with integer entries.
Therefore, the set $B_n^{\ast}(0,1)^2$ is a $2$-dimensional parallelogram equipped
with four corner points belonging to $\zz^2$ and $|B_n^{\ast}(0,1)^2| =
|\det{B_n^{\ast}}| =
b_n$. In order to estimate the right-hand side of \eqref{f54} we will cover the
set $B_n^{\ast}(0,1)^2$ by $G = \bigcup_{i=1}^m
(k^i + [0,1]^2)$ with properly chosen integer points $k^i, i=1,...,m$. By
employing the periodicity of $\eta_s$ this yields
\begin{equation}\label{f55}
    \frac{1}{b_n}\int\limits_{B_n^{\ast}(0,1)^2}|\eta_s(x)|\,dx \leq
\frac{m}{b_n} \int\limits_{\tor^2}|\eta_s(x)|\,dx =
\frac{m}{b_n}\|\eta_s\|_1\,.
\end{equation}
Thus, the problem boils down to bounding the number $m$ properly, i.e., by
$cb_n$, where $c$ is a universal constant not depending on $n$. Since,
$B^{\ast}_n(0,1)^2$ is determined by four integer corner points, the length of
each face is at least $\sqrt{2}$ for all $n$. Therefore, independently
of $n$ we need $\ell$ parallel translations $p^i + B^{\ast}_n(0,1)^2$,
$i=1,...,\ell$, where the $p^i$ are integer multiples of the corner points of
$B^{\ast}_n(0,1)^2$ to floor a part $F = \bigcup_{i=1}^{\ell} (p^i +
B^{\ast}_n(0,1)^2$ of the plane $\re^2$ which contains all squares $k+[0,1]^2$
satisfying ($k+[0,1]^2) \cap B_n^{\ast}(0,1)^2 \neq
\emptyset$. By comparing the area we obtain $m \leq |F| = \ell b_n$,
where $\ell$ is universal. Using \eqref{f55} we obtain finally
$\|\eta_s(B_n^{\ast}\cdot)\|_1 \lesssim \|\eta_s\|_1$\,.\\\\
{\em Step 4.} In the previous steps we proved \eqref{f48} in case $p=1$ and
$p=\infty$. What remains is a consequence of the following elementary
estimate. If $1<p<\infty$, then
$$
    \|\chi_s\|_p = \Big(\int\limits_{\tor^2}
|\chi_s(x)|^{p-1}|\chi_s(x)|\,dx\Big)^{1/p} \leq
\|\chi_s\|_\infty^{1-1/p}\cdot \|\chi_s\|_1^{1/p}\,.
$$
The proof is complete.
\end{proof}

Now we are ready to prove the main result, Theorem \ref{mainup}.
From the continuous embedding of $\Ba(\TT^2)$ into $B^\alpha_{p,1}(\TT^2)$ for $0 < \theta <1$, it is enough to prove the theorem for
$1\le \theta \le \infty$.
By \eqref{f56} the integration is given by
$$
    |R_n(f)| = \Big|\sum\limits_{k\in L(n)\setminus \{0\}}
\hat{f}(k)\Big|\,.
$$
For $j\in \N_0^2$ we define $\varphi_j = \varphi_{j_1} \otimes \varphi_{j_2}$,
where $\varphi = \{\varphi_s\}_{s=0}^{\infty}$ is a smooth decomposition of
unity according to Definition \ref{cunity}. By exploiting $\sum_{j\in \N_0^2}
\varphi_j(x) = 1$, $x\in \re^2$, we can rewrite the error as follows
$$
 |R_n(f)| = \Big|\sum\limits_{k\in L(n)\setminus \{0\}} \Big(\sum\limits_{j\in
\zz^2} \varphi_j(k)\Big) \hat{f}(k)\Big| = \Big|\sum\limits_{j\in
\zz^2} \sum\limits_{k\in L(n)\setminus
\{0\}}\varphi_j(k) \hat{f}(k)\Big|\,.
$$
Taking the support of the functions $\varphi_j$ into account, see Definition
\ref{cunity}, we obtain by Lemma \ref{L(n)_1} that there is a constant $c$ such
that  $\sum_{k\in L(n)\setminus \{0\}}\varphi_j(k) \hat{f}(k) = 0$ whenever
$|j|_1 < \log b_n - c$. Furthermore, by using the trigonometric
polynomials $\chi_j$, introduced in Lemma \ref{chis}, we get for $j
\neq 0$ the identity
$$
     \sum\limits_{k\in L(n)\setminus
\{0\}}\varphi_j(k) \hat{f}(k) = \langle \delta_j(f) , \chi_j\rangle\,,
$$
where $\delta_j(f)$ is defined in \eqref{f2}. Indeed, here we use the fact, that
$v_j \equiv 1$ on $\mbox{supp } \varphi_j$.
Hence, we can rewrite the error once again and estimate taking Lemma \ref{chis}
into account
\begin{equation}\label{f60}
  \begin{split}
    |R_n(f)| &= \Big|\sum\limits_{|j|_1 \geq \log b_n - c} \langle
    \delta_j(f), \chi_j\rangle \Big| \leq \sum\limits_{|j|_1 \geq \log b_n -
c}\|\delta_j(f)\|_p\cdot    \|\chi_j\|_{p'}\\
    &\lesssim \sum\limits_{|j|_1 \geq \log b_n -
c}\Big(\frac{2^{|j|_1}}{b_n}\Big)^{1/p}\|\delta_j(f)\|_p
  \end{split}
\end{equation}
with $1/p+1/p'= 1$. Applying H\"older's inequality for $1/\theta + 1/\theta' =
1$ we obtain (see Lemma \ref{Fourier})
\begin{equation}\label{f61}
  \begin{split}
    |R_n(f)| &\lesssim
\|f|B^{\alpha}_{p,\theta}(\tor^2)\|\cdot\Big(\sum\limits_{|j|_1 \geq
J_n}2^{-\alpha
|j|_1\theta'}\big(2^{|j|_1}/b_n)^{\theta'/p}\Big)^{1/\theta'}\\
&\lesssim  b_n^{-1/p} \Big(\sum\limits_{|j|_1
\geq
J_n}2^{-|j|_1(\alpha-1/p)\theta'}\Big)^{
1/\theta'}
\end{split}
\end{equation}
for $f\in U^{\alpha}_{p,\theta}(\tor^2)$, where we put $J_n:=\log b_n-c$. We
decompose the sum on the right-hand side into 3 parts
$$
    \sum\limits_{|j|_1 \geq J_n} \leq  \sum\limits_{\substack{|j|_1 \geq J_n
\\j_i \leq J_n, i=1,2}} + \sum\limits_{\substack{j_1 > J_n \\ j_2\geq 0}} +
\sum\limits_{\substack{j_2>J_n \\ j_1 \geq 0}}\,.
$$
The first sum yields (recall that $\alpha>1/p$)
$$
  \sum\limits_{\substack{|j|_1 \geq J_n
\\j_i \leq J_n, i=1,2}}2^{-|j|_1(\alpha-1/p)\theta'} \lesssim
\sum\limits_{u=J_n}^{\infty} \sum\limits_{j_2 = 0}^{J_n}
2^{-u(\alpha-1/p)\theta'} \lesssim b_n^{-(\alpha-1/p)\theta'}\log b_n\,.
$$
Let us consider the second sum, the third one goes similarly.  We have
$$
      \sum\limits_{\substack{j_1 > J_n \\ j_2\geq
0}}2^{-|j|_1(\alpha-1/p)\theta'} =
\sum\limits_{j_1=J_n}^{\infty}2^{-j_1(\alpha-1/p)\theta'}\sum\limits_{j_2=0}^{
\infty} 2^{-j_2(\alpha-1/p)\theta'} \lesssim b_n^{-(\alpha-1/p)\theta'}.
$$
Putting everything into \eqref{f61} yields finally
$$
    |R_n(f)| \lesssim b_n^{-\alpha}\big(\log b_n\big)^{1/\theta'}  =
b_n^{-\alpha}\big(\log b_n\big)^{1-1/\theta}\,.
$$
Of course, we have to modify the argument slightly in case $\theta = 1$, i.e.,
$\theta' = \infty$. The sum in \eqref{f61} has to be replaced by a supremum.
Then we immediately obtain
$$
    \sup\limits_{|j_1|\geq J_n} 2^{-|j|_1(\alpha-1/p)} \lesssim
b_n^{-(\alpha-1/p)}\,,
$$
which yields
$$
      |R_n(f)| \lesssim b_n^{-\alpha}\,.
$$
Note that we do not have any $\log$-term in this case. The proof is complete.
\hspace*{\fill}\vbox{\hrule height0.6pt\hbox{\vrule height1.3ex%
width0.6pt\hskip0.8ex\vrule width0.6pt}\hrule height0.6pt}

\subsection{Integration of non-periodic functions}

The problem of the optimal numerical integration of non-periodic functions is more involved. The cubature formula
below is a modification of \eqref{f32} involving additional boundary values of the function under consideration.
Let $n\in \N$ and $N = 5b_n-2$ then we put ($X_{b_n}$ is defined in \eqref{eq41})
\begin{equation}\label{eq33}
 \begin{split}
   Q_N(f):= &\frac{1}{b_n}\sum\limits_{(x_i,y_i)\in X_{b_n}}f(x_i,y_i)\\
   &+ \frac{1}{b_n}\sum\limits_{(x_i,y_i)\in X_{b_n}}\Big[\Big(y_i-\frac{1}{2}\Big)\big(f(x_i,0)-f(x_i,1)\big)+
   \Big(x_i-\frac{1}{2}\Big)\big(f(0,y_i)-f(1,y_i)\big)\Big]\\
   &+\Big(\frac{1}{2b_n}-\frac{1}{4}+\frac{1}{b_n}\sum\limits_{(x_i,y_i)\in X_{b_n}}x_iy_i\Big)
   \big(f(0,0)-f(1,0)+f(1,1)-f(0,1)\big)\,.
 \end{split}
\end{equation}
Let us denote by
$$
        R_N(f) := Q_N(f)-I(f)
$$
the cubature error for a non-periodic function $f\in B^{\alpha}_{p,\theta}(\II^2)$ with respect to the method $Q_N$.
The following theorem gives an upper bound for the worst-case cubature error of the method $Q_N$ with respect to the class $U^{\alpha}_{p,\theta}(\II^2)$. 

\begin{theorem} Let $1\leq p,\theta \leq \infty$ and $1/p<\alpha<1+1/p$. Let $b_n$ denote the $n$th Fibonacci number for $n\in \N$, and $N=5b_n-2$. Then we have
\begin{equation}\label{eq51}
        \sup\limits_{f\in U^{\alpha}_{p,\theta}(\II^2)} |R_N(f)| \leq CN^{-\alpha}(\log N)^{(1-1/\theta)_+}\,.
\end{equation}
\end{theorem}

\begin{proof} By \eqref{eq34} we can decompose a function $f\in U^{\alpha}_{p,\theta}(\II^2)$ into
\begin{equation}\label{eq35}
  \begin{split}
    f(x,y) = f_0(x,y) &+ (1-y)f_1(x) + yf_2(x)\\
    &+ (1-x)f_3(y) + xf_4(y)\\
    &+ f(0,0)(1-x)(1-y) + f(1,0)x(1-y) + f(0,1)(1-x)y + f(1,1)xy\,,
  \end{split}
\end{equation}
where
$$
   f_0(x,y) = \frac{1}{4}\sum\limits_{(j_1,j_2)\in \N^2_0} \sum\limits_{m\in D_{j_1}\times D_{j_2}}
   \Delta^{2,2}_{(2^{-j_1-1}, 2^{-j_2-2})}(f,(2^{-j_1}m_1,2^{-j_2}m_2))v_{j_1,m_1}(x)v_{j_2,m_2}(y)\,,
$$ and
\begin{equation}
\begin{split}
  f_1(x) &= -\frac{1}{2}\sum\limits_{j\in \N_0} \sum\limits_{m\in D_j}
  \Delta^{2}_{2^{-j-1},1}(f,(2^{-j}m,0))v_{j,m}(x)\,,\\
  f_2(x) &= -\frac{1}{2}\sum\limits_{j\in \N_0} \sum\limits_{m\in D_j}
  \Delta^{2}_{2^{-j-1},1    }(f,(2^{-j}m,1))v_{j,m}(x)\,,\\
  f_3(y) &= -\frac{1}{2}\sum\limits_{j\in \N_0} \sum\limits_{m\in D_j}
  \Delta^{2}_{2^{-j-1},2}(f,(0,2^{-j}m))v_{j,m}(y)\,,\\
  f_4(y) &= -\frac{1}{2}\sum\limits_{j\in \N_0} \sum\limits_{m\in D_j}
  \Delta^{2}_{2^{-j-1},2}(f,(0,2^{-j}m))v_{j,m}(y)\,.
\end{split}
\end{equation}
The functions $f_0,...,f_4$ have vanishing boundary values and, therefore, are periodic functions on 
$\TT^2$. Moreover, Lemmas \ref{conv} and
\ref{contvsdisc} (and its univariate version) imply that $f_0 \in U^{\alpha}_{p,\theta}(\tor^2)$ and 
$f_1,...,f_4 \in U^{\alpha}_{p,\theta}(\tor)$\,.
Note that at this point the condition $1/p<\alpha<1+1/p$ is required.
Applying the cubature formula $Q_N$ to \eqref{eq35} yields
\begin{equation}\label{eq36}
  \begin{split}
    Q_Nf = Q_Nf_0 &+ Q_N[(1-y)f_1(x)] + Q_N[yf_2(x)]\\
    &+ Q_N[(1-x)f_3(y)] + Q_N[xf_4(y)]\\
    &+ f(0,0)Q_N[(1-x)(1-y)] + f(1,0)Q_N[x(1-y)]\\
    &+ f(0,1)Q_N[(1-x)y ]+ f(1,1)Q_N[xy]\,.
  \end{split}
\end{equation}
Taking the definition of $Q_N$ in \eqref{eq33} into account we deduce that
\begin{equation}\label{eq43}
    Q_Nf_0 = \frac{1}{b_n}\sum\limits_{(x_i,y_i)\in X_{b_n}} f(x_i,y_i)
\end{equation}
and
\begin{equation}\label{eq44}
  \begin{split}
    Q_N[(1-y)f_1(x)] &= \frac{1}{b_n} \sum\limits_{(x_i,y_i)\in X_{b_n}}(1-y_i)f_1(x_i) + \frac{1}{b_n}\sum\limits_{(x_i,y_i)\in X_{b_n}}(y_i-1/2)f_1(x_i)\\
    &= \frac{1}{2b_n}\sum\limits_{(x_i,y_i)\in X_{b_n}} f_1(x_i)\,.
  \end{split}
\end{equation}
Analogously, we obtain
\begin{equation}\label{eq45}
  \begin{split}
    &Q_N[yf_2(x)] = \frac{1}{2b_n}\sum\limits_{(x_i,y_i)\in X_{b_n}} f_2(x_i)~,~Q_N[(1-x)f_3(y)] = \frac{1}{2b_n}\sum\limits_{(x_i,y_i)\in X_{b_n}} f_3(y_i)\,,\\
    &Q_N[xf_4(y)] = \frac{1}{2b_n}\sum\limits_{(x_i,y_i)\in X_{b_n}} f_4(y_i)\,.
  \end{split}
\end{equation}
Additionally, we get
\begin{equation}\label{eq42}
  \begin{split}
    f(1,1)Q_N[xy] =& f(1,1)\Big[\Big(\frac{1}{2b_n}-\frac{1}{4}+\frac{1}{b_n}\sum\limits_{(x_i,y_i)\in X_{b_n}}x_iy_i\Big)\\
    &~~~~~~~~~+\frac{1}{b_n}\sum\limits_{(x_i,y_i)\in X_{b_n}}\big(x_iy_i + (1/2-y_i)x_i+(1/2-x_i)y_i\big)\Big]\\
    =&f(1,1)\Big[\frac{1}{2b_n}-\frac{1}{4}+\frac{1}{2b_n}\sum\limits_{(x_i,y_i)\in X_{b_n}}x_i + \frac{1}{2b_n}\sum\limits_{(x_i,y_i)\in X_{b_n}}y_i\Big]\,.
  \end{split}
\end{equation}
It turns out that
\begin{equation}\label{eq38}
   \frac{1}{b_n}\sum\limits_{(x_i,y_i)\in X_{b_n}}x_i = \frac{1}{b_n}\sum\limits_{(x_i,y_i)\in X_{b_n}}y_i = \frac{1}{2}-\frac{1}{2b_n}\,.
\end{equation}
In fact,
$$
   \frac{1}{b_n}\sum\limits_{(x_i,y_i)\in X_{b_n}}x_i = \frac{1}{b_n^2}\sum\limits_{\mu=0}^{b_n-1} \mu = \frac{b_n(b_n-1)}{2b_n^2} = \frac{1}{2}-\frac{1}{2b_n}\,.
$$
Furthermore,
\begin{equation}\label{eq40}
    \frac{1}{b_n}\sum\limits_{(x_i,y_i)\in X_{b_n}}y_i = \frac{1}{b_n}\sum\limits_{\mu=1}^{b_n-1}\Big\{\mu\frac{b_{n-1}}{b_n}\Big\}
    =\frac{1}{b_n}\sum\limits_{\mu=1}^{b_n -1}\Big[\frac{1}{2}-\frac{1}{2\pi i}\sum\limits_{k\in \zz} \frac{e^{2\pi i kx}}{k}\Big]_{x=\mu b_{n-1}/b_n}\,,
\end{equation}
where we used the identity
$$
        x = \frac{1}{2}-\frac{1}{2\pi i}\sum\limits_{k\in \zz} \frac{e^{2\pi ikx}}{k}\quad,\quad x\in \tor\setminus\{0\}\,.
$$
Thus, \eqref{eq40} yields
\begin{equation}\label{eq37}
    \frac{1}{b_n}\sum\limits_{(x_i,y_i)\in X_{b_n}}y_i = \frac{1}{2}-\frac{1}{2b_n} - \lim\limits_{N\to\infty}
    \frac{1}{2\pi i}\sum\limits_{1\leq |k|\leq N}\frac{1}{k}\frac{1}{b_n}\sum\limits_{\mu=1}^{b_n-1}e^{2\pi i k \mu\frac{b_{n-1}}{b_n}}\,.
\end{equation}
Since $b_{n-1}$ and $b_n$ do not have a common divisor we have
$$
    \frac{1}{b_n}\sum\limits_{\mu=1}^{b_n-1}e^{2\pi i k \mu\frac{b_{n-1}}{b_n}} = \left\{\begin{array}{rcl}
        1-\frac{1}{b_n}&:& k/b_n \in \zz\,,\\
        -\frac{1}{b_n}&:&\mbox{otherwise}\,.
    \end{array}\right.
$$
The important thing is that $\frac{1}{b_n}\sum_{\mu=1}^{b_n-1}e^{2\pi i k \mu\frac{b_{n-1}}{b_n}}$ does not depend on $k$. Therefore, the sum on the right-hand side
in \eqref{eq37} vanishes and we obtain \eqref{eq38}\,. Hence, \eqref{eq42} simplifies to
$$
    f(1,1)Q_N[xy] = \frac{1}{4}f(1,1)\,.
$$
In the same way we obtain
\begin{equation}\label{eq46}
  \begin{split}
    &f(0,0)Q_N[(1-x)(1-y)] = \frac{1}{4}f(0,0)\,,\\
    &f(1,0)Q_N[x(1-y)] = \frac{1}{4}f(1,0)\,,\\
    &f(0,1)Q_N[(1-x)y] = \frac{1}{4}f(0,1)\,.
  \end{split}
\end{equation}
Let us now estimate the error $|R_N(f)| = |I(f)f-Q_Nf|$. By triangle inequality we obtain
\begin{equation}\label{eq50}
  \begin{split}
    |I(f)-Q_Nf| \leq& |I(f_0)-Q_N(f_0)|\\
    &+|I[(1-y)f_1(x)]-Q_N[(1-y)f_1(x)]|+|I[yf_2(x)]-Q_N[yf_2(x)]|\\
    &+|I[(1-x)f_3(y)]-Q_N[(1-x)f_3(y)]| + |I[xf_4(y)]-Q_N[xf_4(y)]|\,.
  \end{split}
\end{equation}
Note that the remaining error terms disappear, since by \eqref{eq46} the last four functions in the decomposition \eqref{eq36} are integrated exactly. Since $f_0 \in U^{\alpha}_{p,\theta}(\tor^2)$ we obtain by Theorem \ref{mainup} the bound
$$
   |I(f_0)-Q_N(f_0)| \lesssim b_n^{-\alpha}(\log b_n)^{(1-1/\theta)_+} \lesssim N^{-\alpha} (\log N)^{(1-1/\theta)_+}\,.
$$
Let us now estimate the second summand in \eqref{eq50}. By using \eqref{eq44} and the fact that $f_1 \in U^{\alpha}_{p,\theta}(\tor)$ we see
\begin{equation}\nonumber
    |I[(1-y)f_1(x)]-Q_N[(1-y)f_1(x)]| = \frac{1}{2}\Big|\frac{1}{b_n}\sum\limits_{(x_i,y_i)\in X_{b_n}}f(x_i)-I(f_1)\Big| \lesssim b_n^{-\alpha} \lesssim N^{-\alpha}\,.
\end{equation}
Finally, by using \eqref{eq45} we can estimate the remaining terms in \eqref{eq50} in a similar fashion. Altogether we end up with \eqref{eq51} which concludes the proof. 
\end{proof}

\section{Lower bounds for optimal cubature}
\label{lower}

This section is devoted to lower bounds for the $d$-variate integration
problem. The following theorem represents the main result of this section.

\begin{theorem} \label{Theorem[lower bound]}
Let $\ 0 < p, \theta \le \infty$ and $ \alpha > 1/p$. Then we have
\begin{equation} \nonumber
\Int_n(\UaT)
\ \gtrsim \
n^{-\alpha} \log^{(d-1)(1 - 1/\theta)_+} n.
\end{equation}

\end{theorem}
\begin{proof} Observe that
\begin{equation} \label{[i_n(W)>]}
\Int_n(F_d)
\ \ge \
\inf_{X_n=\{x^j\}_{j=1}^n \subset \TTd} \ \ \sup_{f \in F_d: \ f(x^j) = 0, \
j=1,...,n} \ |I(f)|\,.
\end{equation}
Fix an integer $r \ge 2$ so that $\alpha < \min\{r, r - 1 + 1/p\}$ and let $\nu
\in \NN$ be given by the condition $2^{\nu-1} < r \le 2^{\nu}$.
We define the function $g$ on $\RR$ by
\begin{equation} \nonumber
g(x)
:= \
N(2^\nu x).
\end{equation}
Notice that $g$ vanishes outside the interior of the closed interval $\II$.
Let the univariate functions $g_{k,s}$ on $\II$ defined for $k \in {\ZZ}_+, \ s
\in S^1(k)$, by
\begin{equation} \label{[g_{k,s}(d=1)]}
g_{k,s}(x):= \ g(2^k x - s),
\end{equation}
and the $d$-variate functions $g_{k,s}$ on $\IId$  for $k \in \ZZdp, \ s \in
S^d(k)$, by
\begin{equation} \label{[g_{k,s}(d>1)]}
g_{k,s}(x):=  \ \prod_{i=1}^d g_{k_i, s_i}( x_i),
\ k \in {\ZZ}^d_+, \ s \in {\ZZ}^d,
\end{equation}
where 
\begin{equation} \label{[S^d(k)]}
S^d(k) := \{s \in \ZZdp: 0 \le s_j \le 2^{k_j}-1, \ j \in [d]\}.
\end{equation}
We define the open $d$-cube $I_{k,s} \subset \IId$ for $k \in \ZZdp, \ s \in
S^d(k)$, by
\begin{equation} \label{[I_{k,s}]}
I_{k,s}  := \{x \in \IId: 2^{-k_j}s_j < x_j < 2^{-k_j}(s_j + 1), \ j \in [d]\}.
\end{equation}
It is easy to see that every function $g_{k,s}$ is nonnegative in $\IId$ and
vanishes in
$\IId \setminus I_{k,s} $. Therefore, we can extend $g_{k,s}$ to $\RRd$ so that
the extension is $1$-periodic in each variable. We denote this
$1$-periodic extension  by $\tilde{g}_{k,s}$.\\
Let $n$ be given and and $X_n = \{x^j\}_{j=1}^n$ be an arbitrary set of $n$
points in $\TTd$. Without loss of generality we can assume that $n=2^m$. Since $I_{k,s}
\cap I_{k,s'} = \emptyset$ for $s \not= s'$, and $|S^d(k)| = 2^{|k|_1}$, for
each $k \in \ZZdp$ with $|k|_1 = m + 1$, there is $S_*(k) \subset S^d(k)$ such
that $|S_*(k)| = 2^m$ and
$I_{k,s}  \cap X_n = \emptyset$ for every $s \in S_*(k)$. Consider the following function on $\TTd$
\begin{equation} \nonumber
g^*
:= \
C 2^{-\alpha m} m^{-(d-1)/\theta} \sum_{|k|_1 = m+1}\ \sum_{s \in S_*(k)} \tilde{g}_{k,s}.
\end{equation}
By the equation $\tilde{g}_{k,s}(x) = N_{k+\nu {\bf 1},s}(x), \ x \in \IId$, together with Lemma \ref{lemma[InverseRepresTheorem]} and \eqref{eq:StabIneq} we can verify
that
\begin{equation} \label{[|g^*|_{Ba}]}
\|g^*\|_{\Ba}
\ \asymp \
C,
\end{equation}
and
\begin{equation} \label{[|g^*|_1]}
\|g^*\|_1
\ \asymp \
C 2^{-\alpha m} m^{(d-1)(1 - 1/\theta)}.
\end{equation}
By \eqref{[|g^*|_{Ba}]} we can choose the constant $C$ so that $g^* \in \Ua$.
From the construction and the above properties of the function $g_{k,s}$ and the set
$I_{k,s} $, we have  $g^*(x^j) = 0$ for $j =1,...,n$. Hence, by \eqref{[i_n(W)>]} and
\eqref{[|g^*|_1]} we obtain
\begin{equation} \nonumber
\Int_n(\UaT)
\ \ge \ |I(g^*)|
\ = \
\|g^*\|_1
\ \asymp \
n^{-\alpha} \log^{(d-1)(1 - 1/\theta)} n.
\end{equation}
This proves the theorem for the case $\theta \ge 1$.\\
\indent To prove the theorem for the case $\theta < 1$, we take $k \in \ZZdp$ with $|k|_1 = m + 1$, and consider the function on $\TTd$
\begin{equation} \nonumber
g_k:= \
C' 2^{-\alpha m}  \sum_{s \in S_*(k)} \tilde{g}_{k,s}.
\end{equation}
Similarly to the argument for $g^*$, we can choose the constant $C'$ such that $g_k \in \Ua$ and
\begin{equation} \label{[|g^{(k)}|_1]}
\|g_k\|_1
\ \asymp \
2^{-\alpha m}.
\end{equation}
We have  $g_k(x^j) = 0$ for $j =1,...,n$. Hence, by \eqref{[i_n(W)>]} and
\eqref{[|g^{(k)}|_1]} we obtain
\begin{equation} \nonumber
\Int_n(\UaT)
\ \ge \ |I(g_k)|
\ = \
\|g_k\|_1
\ \asymp \
n^{-\alpha}.
\end{equation}
The proof is complete.
\end{proof}

Let us conclude this section with presenting the correct asymptotical behavior of the optimal cubature error in the bivariate case, i.e., in periodic and non-periodic Besov spaces $B^{\alpha}_{p,\theta}(\GG^2)$ with $\GG = \II,\tor$. From Theorem \ref{Theorem[lower bound]} together with Theorem \ref{mainup} we obtain 
\begin{corollary}\label{cor1} Let  $1\leq p \leq \infty$, $0 < \theta \leq \infty$. Then we have the following. \\
{\em (i)} For $\alpha>1/p$,
$$
 \Int_n(\Ua(\TT^2))
\ \asymp \ n^{-\alpha}(\log n)^{(1-1/\theta)_+}.
$$
{\em (ii)} For $1/p<\alpha<1+1/p$,
$$
 \Int_n(\Ua(\II^2))
\ \asymp \ n^{-\alpha}(\log n)^{(1-1/\theta)_+}.
$$
\end{corollary}

\begin{remark}\label{Markh1} \rm Note that the so far best known upper bound for $\Int_n(\Ua(\TT^2))$ was restricted to $\alpha < 2$, see \cite[Thm.\ 4.7]{Ul12_3}. Corollary \ref{cor1} shows in addition that the lower bound in Theorem \ref{Theorem[lower bound]} is sharp in case $d=2$. We conjecture that this is also the case if $d>2$. In fact, Markhasin's results \cite{Ma13_2,Ma13_1,MaDiss12} in combination with Theorem 
\ref{Theorem[lower bound]} verify this conjecture in case of the smoothness $\alpha$ being less or equal to $1$. What happens in case $\alpha>1$ and $d>2$ is open. However, there is some hope for answering this question in case $1/p < \alpha < 2$ by proving a multivariate version of the main result in \cite[Thm.\ 4.7]{Ul12_3}, where Hammersley points have been used. In contrast to the Fibonacci lattice, which has certainly no proper counterpart in $d$ dimensions, this looks possible.
\end{remark}

\section{Cubature and sampling on Smolyak grids}
\label{smolyak}

In this section, we prove asymptotically sharp upper and lower bounds for the error of optimal cubature on Smolyak grids. Note 
that the degree of freedom in the cubature method reduces to the choice of the weights in \eqref{1.1}, the grid remains fixed. Recall the definition of the sparse Smolyak grid $G^d(m)$ given in \eqref{eq53}. It turns out that the upper bound can be obtained directly from results in \cite{Di11, SU07, SU11} on sampling recovery on $G^d(m)$ for $\UaG$. The lower bounds for both the errors of optimal sampling recovery and optimal cubature on $G^d(m)$ will be proved by constructing test functions similar to those constructed in the proof of Theorem \ref{Theorem[lower bound]}.

For a family $\Phi = \{\varphi_\xi\}_{\xi \in G^d(m)}$ of functions we define the linear sampling algorithm
$S_m(\Phi,\cdot)$ on Smolyak grids $G^d(m)$ by
\begin{equation*}
S_m(\Phi,f)
\ = \
\sum_{\xi \in G^d(m)} f(\xi) \varphi_\xi\,.
\end{equation*}
Let us introduce the quantity of optimal sampling recovery $r^s_n(F_d)_q$ on Smolyak grids $G^d(m)$ with respect to the function class $F_d$ by
\begin{equation} \label{def:r^s_n}
r^s_n(F_d)_q
\ := \ \inf_{|G^d(m)| \le n, \, \Phi} \  \sup_{f \in F_d} \, \|f - S_m(\Phi,f)\|_q.
\end{equation}
The upper index $s$ indicates that we restrict to Smolyak grids here. 

\begin{theorem} \label{Theorem:UpperBound[r^s_n]}
Let $\ 0 < p, q, \theta \le \infty$ and $1/p < \alpha$. Then we have for $p \ge q$,
\begin{equation*}
r^s_n(\UaG)_q
 \lesssim
\begin{cases}
(n^{-1} \log^{d-1}n)^\alpha, \ & \theta \le \min\{q,1\}, \\[1.5ex]
(n^{-1} \log^{d-1}n)^\alpha (\log^{d-1}n)^{ 1/q - 1/\theta}, \ & \theta > \min\{q,1\}, \ q \le 1, \\[1.5ex]
(n^{-1} \log^{d-1}n)^\alpha (\log^{d-1}n)^{ 1 - 1/\theta}, \ & \theta > \min\{q,1\}, \ q > 1,
\end{cases}
\end{equation*}
and for $p < q$,
\begin{equation*}
r^s_n(\UaG)_q
\lesssim
\begin{cases}
(n^{-1} \log^{d-1}n)^{\alpha - 1/p + 1/q}(\log^{d-1}n)^{(1/q - 1/\theta)_+}, \ & q < \infty, \\[1.5ex]
(n^{-1} \log^{d-1}n)^{\alpha - 1/p}(\log^{d-1}n)^{(1 - 1/\theta)_+}, \ &  q = \infty.
\end{cases}
\end{equation*}
\end{theorem}

\begin{proof}
This theorem has been proved in \cite{Di11} for $\GGd = \IId$. A slight modification proves the result also for $\GGd = \TTd$.
\end{proof}

The following theorem establishes lower bounds for $r^s_n(\UaG)_q$.

\begin{theorem} \label{Theorem:LowerBound[r^s_n]}
Let $\ 0 < p, q, \theta \le \infty$,  and $\alpha > 1/p$. Then we have for $p \ge q$,
\begin{equation*}
r^s_n({\UaG})_q
 \ \gtrsim \
(n^{-1} \log^{d-1}n)^\alpha (\log^{d-1}n)^{(1 - 1/\theta)_+}, 
\end{equation*}
and for $p < q$,
\begin{equation*}
r^s_n(\UaG)_q
 \ \gtrsim \
(n^{-1} \log^{d-1}n)^{\alpha - 1/p + 1/q}(\log^{d-1}n)^{(1/q - 1/\theta)_+}.
\end{equation*}
\end{theorem}

\begin{proof} Clearly, it is enough to prove the theorem for $\GGd = \TTd$.
Observe that
\begin{equation} \label{[r^s_n>]}
r^s_n(F_d)_q
\ \ge \ \inf_{|G^d(m)| \le n} \
\sup_{f \in F_d: \ f(\xi) = 0, \ \xi \in G^d(m)} \, \|f \|_q.
\end{equation}
We will use the sets $S^d(k)$, $I_{k,s}$ and the periodic functions $\tilde{g}_{k,s}$ constructed in the proof of Theorem \ref{Theorem[lower bound]} (see \eqref{[g_{k,s}(d=1)]}--\eqref{[I_{k,s}]} and the following definition of $\tilde{g}_{k,s}$). In particular, we have
\begin{equation} \nonumber
\tilde{g}_{k,s}(x):=  \ \prod_{j=1}^d \tilde{g}_{k_j, s_j}( x_j),
\ k \in {\ZZ}^d_+, \ s \in {\ZZ}^d,
\end{equation}
and
\begin{equation} \nonumber
I_{k,s} \ = \ \prod_{j=1}^d I_{k_j,s_j},
\end{equation}
where the univariate functions $\tilde{g}_{k_j, s_j}( x_j)$ is nonnegative in $\II$ and vanishes in
$\II \setminus I_{k_j,s_j}$. Let $m$ be an arbitrary integer such that $|G^d(m)| \le n$. Without loss of generality we can assume that $m$ is the maximum among such numbers. We have
\begin{equation} \label{asymp:[2^m]}
2^m
\ \asymp \ n (\log n)^{-(d-1)}.
\end{equation}
Put $D(m):= \{(k,s): k \in \ZZdp, \ |k|_1 = m, \ s \in S^d(k)\}$. We prove that $\tilde{g}_{k,s}(\xi) = 0$ for every $(k,s) \in D(m)$ and $\xi \in G^d(m)$. Indeed, $(k,s) \in D(m)$ and $\xi = 2^{-k'}s' \in G^d(m)$, then there is $j \in [d]$ such that $k_j \ge k'_j$. Hence, by the construction we have $\tilde{g}_{k_j, s_j}(2^{-k'_j}s'_j) =0$, and consequently, $\tilde{g}_{k,s}(2^{-k'}s') = 0$.

Moreover,  if $0 < \nu \le \infty$, for $(k,s) \in D(m)$,
\begin{equation} \label{asymp:Norm[g_{k,s}]}
\|\tilde{g}_{k,s}\|_\nu
\ \asymp \ 2^{- m/\nu} ,
\end{equation}
with the change to sup when $\nu = \infty$, and
\begin{equation}  \label{asymp:NormSum[g_{k,s}]}
\Big\|\sum_{s \in S^d(k)} \tilde{g}_{k,s}\Big\|_\nu
\ \asymp \ 1.
\end{equation}

Consider the functions
\begin{equation} \label{def:phi_1}
\varphi_1
\ := \
C_1 2^{-\alpha m} \sum_{s \in S^d({\bar k})} \tilde{g}_{{\bar k}, s}
\end{equation}
for some ${\bar k}$ with $|{\bar k}| = m$, and
\begin{equation} \label{def:phi_2}
\varphi_2
\ := \
C_2 2^{- \alpha m} m^{-(d-1)/\theta} \sum_{(k,s) \in D(m)}\tilde{g}_{k,s}.
\end{equation}
By Lemma \ref{lemma[InverseRepresTheorem]}
and \eqref{asymp:NormSum[g_{k,s}]}
we can choose constants $C_i$ so that
$\varphi_i \in \UaT$ for all $m \ge 1$ and $i=1,2$. By the construction we have
 $\varphi_i(\xi) = 0, \ i=1,2$, for every $\xi \in G^d(m)$.
We have by \eqref{asymp:[2^m]} -- \eqref{asymp:NormSum[g_{k,s}]}
\begin{equation} \label{[|varphi_1|_q >]}
r^s_n(\UaT)_q
\  \ge \ \|\varphi_1\|_q
\  \gtrsim \
 2^{-\alpha m}
\ \asymp \
 (n^{-1} \log^{d-1}n)^\alpha
\end{equation}
if $\theta \le 1$, and
\begin{equation} \label{[|varphi_2|_q >]}
\begin{split}
r^s_n(\UaT)_q
\  &\ge \ \|\varphi_2\|_q \ \ge \|\varphi_2\|_{q^*}
\ \gtrsim \
  2^{-\alpha m} m^{(d-1)(1 - 1/\theta)} \\
\ &\asymp \
 (n^{-1} \log^{d-1}n)^\alpha(\log^{d-1}n)^{1 - 1/\theta}
\end{split}
\end{equation}

We take the functions
\begin{equation} \label{def:g_3}
\varphi_3
\ = \
C_3 2^{-(\alpha - 1/p)m} \tilde{g}_{k^*,s*}
\end{equation}
with some $(k^*,s^*) \in D(m)$, and
\begin{equation} \label{def:g_4}
\varphi_4
\ = \
C_4 2^{-(\alpha - 1/p)m} m^{-(d-1)/\theta} \sum_{|k|_1 = m} \tilde{g}_{k,s(k)}
\end{equation}
with some $s(k) \in S^d(k)$.
Similarly to the functions $\varphi_i, \ i =1,2$,  we can choose constants $C_i$ so that
$\varphi_i \in \UaT$, $i=3,4$, by the construction we have
 $\varphi_i(\xi) = 0, \, i=3,4$, for every $\xi \in G^d(m)$.
We have by \eqref{asymp:[2^m]} and \eqref{asymp:Norm[g_{k,s}]}
\begin{equation*}
r^s_n(\UaT)_q
\  \ge \ \|\varphi_3\|_q
\  \gtrsim \
 2^{-(\alpha - 1/p + 1/q)m}
 \ \asymp \
(n^{-1} \log^{d-1}n)^{\alpha - 1/p + 1/q}
\end{equation*}
if $\theta \le q$, and
\begin{equation*}
\begin{split}
r^s_n(\UaT)_q
\  &\ge \ \|\varphi_4\|_q
\  \gtrsim \
 2^{-(\alpha - 1/p + 1/q)m} m^{(d-1)(1/q - 1/\theta)} \\
 \ &\asymp \
(n^{-1} \log^{d-1}n)^{\alpha - 1/p + 1/q}(\log^{d-1}n)^{1/q - 1/\theta}
\end{split}
\end{equation*}
if $\theta > q$.
\end{proof}

\noindent Putting together Theorems \ref{Theorem:UpperBound[r^s_n]} and \ref{Theorem:LowerBound[r^s_n]} we obtain the asymptotically sharp error bounds for optimal sampling on Smolyak grids for $\UaG$.
\begin{corollary} \label{Corollary:Asymp}
Let $\ 0 < p, q, \theta \le \infty$, and $\alpha > 1/p$. Then we have for $p \ge q$,
\begin{equation*}
r^s_n(\UaG)_q
 \ \asymp \
\begin{cases}
(n^{-1} \log^{d-1}n)^\alpha, \, \ & \theta \le \min\{q,1\}, \\[1.5ex]
(n^{-1} \log^{d-1}n)^\alpha (\log^{d-1}n)^{(1 - 1/\theta)_+}, \ & \theta > 1, q \ge 1,
\end{cases}
\end{equation*}
and for $p < q < \infty$,
\begin{equation*}
r^s_n(\UaG)_q
 \ \asymp \
(n^{-1} \log^{d-1}n)^{\alpha - 1/p + 1/q}(\log^{d-1}n)^{(1/q - 1/\theta)_+}.
\end{equation*}
\end{corollary}

\begin{remark}
\rm This corollary has been proved in \cite{Di11} for $1/p < \alpha < 2$ and $\GGd = \IId$.
\end{remark}

\noindent Let us now construct associated cubature formulas. For a family $\Phi = \{\varphi_\xi\}_{\xi \in G^d(m)}$ in $\GGd$, the linear sampling algorithm $S_m(\Phi,\cdot)$  
generates the cubature formula $\Lambda^s_m(f)$ on Smolyak grid $G^d(m)$ by
\begin{equation} \label{def[Lambda^s_m(f)(1)]}
\Lambda^s_m(f)
\ = \
\sum_{\xi \in G^d(m)} \lambda_\xi f(\xi),
\end{equation}
where the vector $\Lambda_m$ of integration weights is given by 
\begin{equation} \label{def[Lambda^s_m(f)(2)]}
\Lambda_m = (\lambda_\xi)_{\xi \in G^d(m)}, \quad \lambda_\xi = \int_{\IId} \varphi_\xi \ dx.
\end{equation}
Hence, it is easy to see that
\begin{equation*}
 |I(f) - \Lambda^s_m(f)|
 \ \le \
 \|f - S_m(\Phi,f)\|_1,
\end{equation*}
and, as a consequence of \eqref{def:r^s_n} and \eqref{def:Int^s_n},
\begin{equation} \label{[i_n<r_n]}
\Int^s_n(F_d)
\ \le \
r^s_n(F_d)_1.
\end{equation}

\noindent The following theorem represents the main result of this section. It states the correct asymptotic of the error of optimal cubature 
on Smolyak grids for $\UaG$. 

\begin{theorem} \label{Corollary:Asump[Int^s_n]}
Let $\ 0 < p, q, \theta \le \infty$ and $\alpha>1/p$.
Then we have
\begin{equation*}
\Int^s_n(\UaG)
\ \asymp  \
n^{-\alpha} (\log^{d-1} n)^{\alpha + (1 - 1/\theta)_+}.
\end{equation*}
\end{theorem}

\begin{proof}
The upper bound is derived from \eqref{[i_n<r_n]} and Theorem \ref{Theorem:UpperBound[r^s_n]} together with 
\eqref{def[Lambda^s_m(f)(1)]}--\eqref{def[Lambda^s_m(f)(2)]} and \eqref{[i_n<r_n]}. To prove the lower bound we employ the inequality
\begin{equation} \label{[Int^s_n>]}
\Int^s_n(F_d)_q
\ \ge \ \inf_{|G^d(m)| \le n} \
\sup_{f \in F_d: \ f(\xi) = 0, \ \xi \in G^d(m)} \, |I(f)|.
\end{equation}
Consider the functions  $\varphi_i, \ i=1,2$, as defined as in \eqref{def:phi_1} and \eqref{def:phi_2}. We have
$|I(\varphi_i)| = \|\varphi_i\|_1, \ i=1,2$. Hence, by \eqref{[Int^s_n>]}, \eqref{[|varphi_1|_q >]} and \eqref{[|varphi_2|_q >]} for $q=1$ we obtain the lower bound.
\end{proof}


\begin{remark} \rm In case $d=2$ the lower bound in Theorem \ref{Corollary:Asump[Int^s_n]} is significantly larger than the bounds provided in Corollary \ref{cor1} for all $\alpha>1/p$. Therefore, cubature formulae based on Smolyak grids can never be optimal for $\Int_n(\Ua(\TT^2))$. We conjecture, that this is also the case in higher dimensions $d>2$. In fact, 
considering Markhasin's results \cite{Ma13_2,Ma13_1,MaDiss12} in combination with Theorem \ref{Corollary:Asump[Int^s_n]} verifies this conjecture in case of the smoothness $\alpha$ being less or equal to $1$. What happens in case $\alpha>1$ and $d>2$ is open. However, there is some hope for answering this question in case $1/p < \alpha < 2$ by proving a multivariate version of the main result in \cite{Ul12_3}. See also Remark \ref{Markh1} above. 
\end{remark}

\begin{remark} \rm
An asymptotically optimal cubature formula on the Smolyak grid is generated by the method described in \eqref{def[Lambda^s_m(f)(1)]}--\eqref{def[Lambda^s_m(f)(2)]} of the optimal sampling algorithm, which indeed exists, see \cite{Di11, SU07, SU11}. 
\end{remark}

\noindent
{\bf Acknowledgments.} The research of Dinh D\~ung is funded by Vietnam National Foundation for Science and Technology Development (NAFOSTED) under Grant No. 102.01-2012.15. Both authors would like to thank Aicke Hinrichs for useful comments and suggestions.

\end{document}